\documentclass[10pt]{amsart}
\usepackage{amsmath,amsthm,amsfonts,amssymb,amscd,flafter,pinlabel}
\usepackage{mathtools}
\usepackage[mathscr]{eucal}
\usepackage{graphics}
 \usepackage[all]{xy}
\usepackage{caption, subcaption}
\usepackage[usenames,dvipsnames]{color}
\usepackage{graphicx}

\usepackage{pgf}
\usepackage{tikz}
\usetikzlibrary{arrows,automata}
\usepackage[latin1]{inputenc}

\usepackage[colorlinks=true, urlcolor=NavyBlue, linkcolor=NavyBlue, citecolor=NavyBlue, pdftitle={Trisections, intersection forms and the Torelli group}, pdfauthor={Peter Lambert-Cole}, pdfsubject={Trisections and Torelli gorup}, pdfkeywords={4-manifolds, Torelli group, intersection form}]{hyperref}

\newtheorem{theorem}{Theorem}[section]
\newtheorem{corollary}[theorem]{Corollary}

\newtheorem{proposition}[theorem]{Proposition}

\newtheorem{lemma}[theorem]{Lemma}

\theoremstyle{definition}
\newtheorem{definition}[theorem]{Definition}

\theoremstyle{definition}
\newtheorem{remark}[theorem]{Remark}

\newcommand\alphas{\mbox{\boldmath$\alpha$}}
\newcommand\betas{\mbox{\boldmath$\beta$}}
\newcommand\gammas{\mbox{\boldmath$\gamma$}}

\newcommand{\cM}{\mathcal{M}}

\newcommand{\del}{\partial}

\newcommand{\ZZ}{\mathbb{Z}}

\newcommand{\cD}{\mathcal{D}}
\newcommand{\cL}{\mathcal{L}}

\newcommand{\cI}{\mathcal{I}}

\newcommand{\cT}{\mathcal{T}}

\newcommand{\cA}{\mathcal{A}}
\newcommand{\cC}{\mathcal{C}}
\newcommand{\cK}{\mathcal{K}}

\newcommand{\cB}{\mathcal{B}}
\newcommand{\cH}{\mathcal{H}}
\newcommand{\cG}{\mathcal{G}}

\newcommand{\CP}{\mathbb{CP}}

\begin{document}

\title[{Trisections, intersection forms, and the Torelli group}]{Trisections, intersection forms and the Torelli group}

\author[P. Lambert-Cole]{Peter Lambert-Cole}
\address{School of Mathematics \\ Georgia Institute of Technology}
\email{plc@math.gatech.edu}
\urladdr{\href{https://www.math.gatech.edu/~plambertcole3/}{https://www.math.gatech.edu/\~{}plambertcole3}}

% \date{\today}
\maketitle

%%%%%%%%%%%%%%%%%%%%%%%%%%%%%%%%%%%%%%%%%%%%%%%%%%%%%%%

%%%%%%%%%%%%%%%%%%%%%%%%%%%%%%%%%%%%%%%%%%%%%%%%%%%%%%%
\begin{abstract}
%%%%%%%%%%%%%%%%%%%%%%%%%%%%%%%%%%%%%%%%%%%%%%%%%%%%%%%

We apply mapping class group techniques and trisections to study intersection forms of smooth 4-manifolds.  Johnson defined a well-known homomorphism from the Torelli group of a compact surface.  Morita later showed that every homology 3-sphere can be obtained from the standard Heegaard decomposition of $S^3$ by regluing according to a map in the kernel of this homomorphism.  We prove an analogous result for trisections of 4-manifolds.  Specifically, if $X$ and $Y$ admit handle decompositions without 1- or 3-handles and have isomorphic intersection forms, then a trisection of $Y$ can be obtained from a trisection of $X$ by cutting and regluing by an element of the Johnson kernel.  We also describe how invariants of homology 3-spheres can be applied, via this result, to obstruct intersection forms of smooth 4-manifolds.  As an application, we use the Casson invariant to recover Rohlin's Theorem on the signature of spin 4-manifolds.

%%%%%%%%%%%%%%%%%%%%%%%%%%%%%%%%%%%%%%%%%%%%%%%%%%%%%%%	
\end{abstract}
%%%%%%%%%%%%%%%%%%%%%%%%%%%%%%%%%%%%%%%%%%%%%%%%%%%%%%%

\section{Introduction}

Every closed, oriented 3-manifold $Y$ admits a Heegaard decomposition into the union of two genus-$g$handlebodies.  Such a decomposition can be obtained from the standard Heegaard splitting of $S^3$ by cutting and regluing by some element of the mapping class group $\cM(\Sigma)$ of a closed, genus-$g$ surface.  When $Y$ is an integral homology 3-sphere, the map can be chosen in the Torelli group $\cI_g$, which consists of mapping classes that act trivially on homology.  This connection has proven enormously useful in understanding both the mapping class group and integral homology 3-spheres.

Trisections of smooth 4-manifolds are analogous to Heegaard splittings in dimension 3.  Specifically, a $(g;k_1,k_2,k_3)$-{\it trisection} $\cT$ of a smooth 4-manifold $X$ is a decomposition $X = Z_1 \cup Z_2 \cup Z_3$ such that
\begin{enumerate}
\item the sector $Z_{\lambda}$ is diffeomorphic to the 4-dimensional 1-handlebody $\natural_{k_{\lambda}} S^1 \times B^3$,
\item each double intersection $H_{\lambda} = Z_{\lambda-1} \cap Z_{\lambda}$ is diffeomorphic to the 3-dimensional 1-handlebody $\natural_{g} S^1 \times B^2$,
\item the triple intersection $\Sigma_g = Z_1 \cap Z_2 \cap Z_3$ is a closed, oriented surface of genus $g$,
\end{enumerate} 
It follows immediately from the definition that $H_{\lambda} \cup_{\Sigma} H_{\lambda+1}$ is a genus $g$ Heegaard splitting of the boundary $\del Z_{\lambda} = \#_{k_{\lambda}} S^1 \times S^2$.  Conversely, take a genus-$g$ handlebody $H_g$ and three homeomorphisms $\{ \phi_{\lambda}: \del H_g \rightarrow \Sigma_g \}$.  Suppose the union
\[Y_{\lambda} = H_g \cup_{\phi_{\lambda}\phi^{-1}_{\lambda+1}} (-H_g)\]
is homeomorphic to $\#_{k_{\lambda}} S^1 \times S^2$ for $\lambda = 1,2,3$.  Then a result of Laudenbach and Po\'enaru \cite{Laudenbach-Poenaru} states that we can build a unique smooth 4-manifold, up to diffeomorphism, by gluing three handlebodies $H_1,H_2,H_3$ to $\Sigma_g$ via the homeomorphisms $\phi_1,\phi_2,\phi_3$ and then capping off with three 4-dimensional 1-handlebodies.  

\pagebreak

Given the analogy between Heegaard splittings and trisections, it is natural to ask: 
\begin{enumerate}
\item Do results about Heegaard splittings and 3-manifolds extend to 4-dimensions?  
\item Can we apply 3-dimensional techniques to answer questions in 4-manifold topology?
\end{enumerate}

Of specific interest in this paper is the work of Birman and Craggs \cite{Birman-Craggs}, Johnson \cite{Johnson-BC,Johnson2,Johnson3}, Casson, and Morita \cite{Morita-1,Morita-2} that relates the Torelli group to invariants of integral homology 3-spheres.  Birman and Craggs used the Rohlin invariant of homology 3-spheres to define a collection of homomorphisms from the Torelli group to $\ZZ/2 \ZZ$.  Johnson later completely classified these homomorphisms and subsequently used these maps, along with his namesake homomorphism, to determine the abelianization of Torelli group.  Casson lifted the Rohlin invariant and defined a famous $\ZZ$-valued invariant of homology 3-spheres.  Morita then applied Johnson's work to reinterpret the Casson invariant in terms of the mapping class group.  In particular, every integral homology 3-sphere can be obtained by taking the standard Heegaard decomposition of $S^3$, cutting along the Heegaard surface, and regluing by an element of $\cK_g$, the kernel of the Johnson homomorphism.  As Johnson showed, this subgroup is precisely the subgroup generated by separating twists.  The change in the Casson invariant can then be computed using the surgery formula.

The main result of this paper is the following 4-dimensional analogue of Morita's result.
\begin{theorem}
\label{thrm:homeo-Johnson}
Let $X$ and $Y$ be homeomorphic, closed smooth 4-manifolds.  Suppose that $X$ and $Y$ admit $(g;k,0,0)$-trisections $\cT_X$ and $\cT_Y$, respectively.  Then $\cT_Y$ can be obtained from $\cT_X$ by cutting and regluing by an element of the Johnson kernel $\cK_g$.
\end{theorem}

Note that the condition that $X$ and $Y$ admit $(g;k,0,0)$-trisections is equivalent to the condition that they are simply-connected and admit perfect Morse functions (see Theorem \ref{thrm:trisection-existence}).

\subsection{Intersection forms}

A second goal is this paper is to describe a way in which, via trisections, mapping class group techniques may be used to obstruct intersection forms.  Let $X$ be a 4-manifold.  Multiplication in the cohomology ring $H^*(X;\ZZ)$ determines a symmetric, bilinear pairing $Q: H^2(X,\ZZ) \times H^2(X,\ZZ) \rightarrow H^4(X,\ZZ) \cong \ZZ$.  When $X$ is simply-connected, it is a result of Wall that $Q$ determines the homotopy-type of $X$ and a result of Freedman that $Q$ determines the homeomorphism type of $X$. Every unimodular quadratic form over $\ZZ$ is the intersection form of a topological 4-manifold.  However, this is not true for smooth 4-manifolds and determining which quadratic forms can be realized is an important open problem in 4-manifold theory.  

There is an extra subtlety that distinguishes trisections from Heegaard splittings.  It is not true that an arbitrary choice of gluing maps $\{\phi_{\lambda}\}$ determines a closed 4-manifold.  The pairwise union $Y_{\lambda} = H_g \cup_{\phi_{\lambda}\phi^{-1}_{\lambda+1}} (-H_g)$ may not be homeomorphic to $\#_{k_{\lambda}} S^1 \times S^2$, let alone have the same integral homology.  This gluing data only determines a more general {\it Heegaard triple} and only when each $Y_{\lambda}$ has the correct homeomorphism type does this construction describe a unique, closed 4-manifold.  When $Y_1$ is homeomorphic to $\#_k S^1 \times S^2$ and $Y_2$ and $Y_3$ are integral homology 3-spheres, we refer to the triple $\{\phi_{\lambda}\}$ as a {\it pseudotrisection}.

While this appears unfortunate, in fact it is an opportunity to show that some quadratic forms cannot be realized as the intersection form of any closed, smooth, oriented 4-manifold.  Let $\{\phi_{\lambda}\}$ be a triple of gluing maps.  Each map induces a map on homology $(\phi_{\lambda}^{-1})_*: H_1(\Sigma) \rightarrow H_1(H_g)$.  The kernel of this map is a $g$-dimensional subspace $L_{\lambda}$ that is Lagrangian with respect to the intersection form on $H_1(\Sigma)$.  As shown by Feller, Klug, Schirmer and Zemke \cite{FKSZ}, the triple of Lagrangian subspaces determines the cohomology ring of the induced 4-manifold.  Moreover, even if the triple $\{\phi_{\lambda}\}$ is only a pseudotrisection, we can formally compute the `intersection form' using the same algebraic formula.

We then have a strategy to obstruct $Q$ as the intersection form of smooth 4-manifold.  Take the set of all pseudotrisections $\{\phi_{\lambda}\}$ whose `intersection form' is $Q$.  As in Theorem \ref{thrm:homeo-Johnson}, any two pseudotrisections with the same `intersection form' are related by an element of the Johnson kernel.  To exclude $Q$, it now suffices to show that after cutting and regluing by any element $\rho$ in $\cK_g$, at least one of the resulting 3-manifolds $Y_{\lambda,\rho}$ is always a nontrivial homology 3-sphere.

\subsection{Rohlin's Theorem}

A standard fact in 4-manifold topology is Rohlin's Theorem.

\begin{theorem}[Rohlin \cite{Rohlin}]
\label{thrm:Rohlin}
Let $X$ be a closed, spin 4-manifold.  Then $\sigma(X) = 0 \text{ mod } 16$.
\end{theorem}

There exist several alternate proofs of this result in the literature (see \cite{Kervaire-Milnor,Freedman-Kirby,Kirby-book,Lawson-Michelsohn,Kirby-Melvin}).    We present a new proof here using the Casson invariant $\lambda$ of homology 3-spheres, whose reduction mod 2 is better known as the Rohlin invariant $\mu$.  A standard approach uses Theorem \ref{thrm:Rohlin} to show that $\mu$ is well-defined.  However, using trisections we can reverse the logical causality and apply the mod 2 Casson invariant to exclude spin 4-manifolds whose signature is $8 \text{ mod }16$.

By standard surgery theory techniques, every spin 4-manifold is spin-cobordant to a simply-connected 4-manifold $X$ that has indefinite intersection form and a handle decomposition without $1-$ and $3$-handles (Lemma \ref{lemma:spin-cobordant}).   Since $X$ is spin, the $2^{\text{nd}}$ Stiefel--Whitney class $w_2(X)$ vanishes and so the intersection form is even.  By the classification of unimodular, indefinite and even symmetric bilinear forms, the intersection form $Q_X$ is isomorphic to $m E_8 \oplus n H$, where a negative value of $m$ corresponds to summands of $-E_8$.  Rohlin's Theorem is then equivalent to the statement that $m$ must be even.  Thus, it suffices to prove the following theorem.

\begin{theorem}
\label{thrm:gsc-Rohlin}
Suppose that  $X$ be a spin 4-manifold that admits a $(g;k,0,0)$-trisection.  Then the intersection form of $X$ has the form
\[Q_X \cong 2m E_8 \oplus n H.\]
for some integer $m$.
\end{theorem}

The first step is find a triple $\{\phi_{\lambda}\}$ of maps that determine $(8;0,0,0)$ pseudotrisection with intersection form $E_8$ (Figure \ref{fig:pseudoE8}).  The resulting 3-manifold $Y_3$ can be immediately identified as the Poincar\'e homology sphere, whose Casson invariant is $-1$.  In particular, since $Y_2$ is homeomorphic to $S^3$, we see that
\[\mu(Y_2) + \mu(Y_3) = 1 \text{ mod } 2\]
Then, we use the surgery formulas for the Casson invariant to prove that
\[\mu(Y_{2,\rho}) + \mu(Y_{3,\rho}) = 1 \text{ mod } 2\]
for all $\rho \in \cK_{8}$. More generally, if we start with a $(g;k,0,0)$ pseudotrisection with intersection form $m E_8 \oplus n H$, we show that
\[\mu(Y_{2,\rho}) + \mu(Y_{3,\rho}) = m \text{ mod }2\]
for any $\rho \in \cK_{g}$.  This suffices to prove Theorem \ref{thrm:gsc-Rohlin} and therefore Rohlin's Theorem.

It would be interesting to apply a stronger invariant of homology 3-spheres, in order to obstruct other intersection forms.  A naive hope is that, since the Casson invariant counts $SU(2)$ representations of the fundamental group and Donaldson invariants count $SU(2)$ instantons, one might recover Donaldson's Diagonalization theorem \cite{Donaldson}.  However, the full Casson invariant does not appear to add more information than its mod 2 version.  For example, it does not obstruct the intersection forms $E_8 \oplus E_8$ or $E_8 \oplus \langle 1 \rangle$, which are definite but not diagonal and therefore excluded by Donaldson.

\subsection{Acknowledgements}

I would like to thank Dan Margalit for his excellent class on the Torelli group.  I would also like to thank Justin Lanier, Agniva Roy and Trent Schirmer for helpful discussions.

\section{Torelli group and the Johnson homomorphism}

\subsection{Torelli group}

Let $\Sigma_{g}$ be a closed, oriented surface of genus $g$ and let $\Sigma_{g,1}$ denote a compact, oriented surface of genus $g$ and 1 boundary component.  Let $\cM_g$ denote the mapping class group of $\Sigma_g$ and let $\cI_{g}$ denote the Torelli group of $\Sigma_g$.  It is the subgroup of $\cM_g$ consisting of classes of homeomorphisms of $\Sigma_g$ to itself that induce the trivial homomorphism on homology.  Let $\cM_{g,1}$ be the mapping class group of $\Sigma_{g,1}$, consisting of isotopy classes of homeomorphisms of $\Sigma_{g,1}$ that fix the boundary pointwise.  Let $\cI_{g,1}$ denote the Torelli group of $\Sigma_{g,1}$. For technical reasons, it is often convenient to work with $\cI_{g,1}$ instead of $\cI_g$.  The groups are related by the exact sequence
\[1 \rightarrow \pi_1(UT\Sigma_g) \rightarrow \cI_{g,1} \rightarrow \cI_g \rightarrow 1\]
where $UT \Sigma_g$ denote the unit tangent bundle of $\Sigma_g$.  Let $\cK_{g}$ and $\cK_{g,1}$ be the subgroups of $\cI_g$ and $\cI_{g,1}$ generated by Dehn twists on separating curves.

The following is a basic fact in the theory of mapping class groups.

\begin{proposition}
\label{prop:cut-system-torelli}
Let $\alphas = (\alpha_1,\dots,\alpha_g)$ be a collection of disjoint, simple closed curves representing linearly independent classes in $H_1(\Sigma)$.  Let $\alphas' = (\alpha'_1,\dots,\alpha'_g)$ be a second collection of disjoint, simple closed curves satisfying $[\alpha_i] = [\alpha'_i]$ for all $i=1,\dots,g$.  There exists an element $\phi \in \cI_g$ such that  $\phi(\alpha_i) =\alpha'_i$ for all $i = 1,\dots,g$.
\end{proposition}

\begin{proof}
We sketch the proof; more details can be found in \cite[Chapter 6.3.2]{FM-primer}.  The cut systems $\alphas,\alphas'$ can be extended to geometric symplectic bases $\{\alphas,\betas\}$ and $\{\alphas',\betas'\}$ such that $[\beta_i] = [\beta'_i]$.  Each complement $\Sigma \smallsetminus (\alphas \cup \betas)$ and $\Sigma \smallsetminus (\alphas' \cup \betas')$ is a sphere with $g$ boundary components.  We can choose a homeomorphism between the spheres that identifies $\alpha_i$ with $\alpha'_i$ and $\beta_i$ with $\beta'_i$ and extends to a homeomorphism $\phi$ of $\Sigma$.  The induced map on homology sends $[\alpha_i]$ to $[\alpha'_i]$ and $[\beta_i]$ to $[\beta'_i]$.  In other words, $\phi$ acts trivially on homology and is therefore in the Torelli group.
\end{proof}

\subsection{Johnson homomorphism}

Let $\Gamma = \Gamma_0$ denote the fundamental group of $\Sigma_{g,1}$, with a basepoint chosen in the boundary, let $\Gamma_k = [\Gamma,\Gamma_{k-1}]$ denote the $k^{\text{th}}$ term in the lower central series for $\Gamma$ and let $N_k = \Gamma/\Gamma_{k}$ be the $k^{\text{th}}$ nilpotent quotient. Note that $N_1 = \Gamma/ [\Gamma,\Gamma] \cong H_1(\Sigma_{g,1};\ZZ)$.  There is an exact sequence of groups
\[ 1 \rightarrow \cL_{k+1} \rightarrow N_{k+1} \rightarrow N_k \rightarrow 1\]
where $\cL_{k+1}$ is the center of $N_{k+1}$.  There is a natural action of $\cM_{g,1}$ on $\Gamma$ and the canonical morphism $\cM_{g,1} \rightarrow Aut(\Gamma)$ is injective.  The subgroups $\Gamma_k$ are characteristic and therefore preserved by any automorphism of $\Gamma$.  Thus, we obtain a map $\cM_{g,1} \rightarrow Aut(N_k)$ and denote its kernel by $\cM_{g,1}(k)$.  The Torelli group $\cI_{g,1}$ is precisely $\cM_{g,1}(1)$.  There are sequence of homomorphisms
\[\tau_k: \cM_{g,1}(k) \rightarrow \text{Hom}(N_1,\cL_{k+1})\]
For $k = 1$, the homomorphism
\[\tau:  \cI_{g,1} \cong \cM_{g,1}(1) \rightarrow \text{Hom}(N_1,\cL_2) \cong \wedge^3 H_1(\Sigma)\]
is known as the {\it the} Johnson homomorphism.  

\begin{theorem}[\cite{Johnson2}]
The kernel of $\tau$ is $\cK_{g,1}$.
\end{theorem}

Recall that a {\it bounding pair} consists of two disjoint, homologous curves $d,d'$ and the corresponding bounding pair map is $T_d T^{-1}_{d'}$ is an element of the Torelli group.  The genus of a bounding pair is the genus of the subsurface cut off by $d \cup d'$. A {\it $k$-chain} $(c_1,\dots,c_k)$ is a collection of simple closed curves such that $i(c_j,c_{j+1}) = 1$ and $i(c_j,c_k) = 0$ for $| j - k| > 1$.  If $k$ is odd, then the boundary of a neighborhood of $c_1 \cup \cdots \cup c_k$ is a bounding pair of genus $\frac{1}{2}(k-1)$.  If $k$ is even, then the boundary of a neighborhood of $c_1 \cup \cdots \cup c_k$ is a separating curve that cuts off a subsurface of genus $\frac{1}{2}k$.

\begin{lemma}
\label{lemma:Johnson-3chain}
Let $(a,b,c)$ be a 3-chain and let $T_{d}T^{-1}_{d'}$ be the corresponding bounding pair.  Then
\[\tau(T_{d}T^{-1}_{d'}) = [a] \wedge [b] \wedge [c]\].
\end{lemma}

To understand equivalences between Heegaard splittings and between trisections, we are interested in bounding pair maps that extend across a handlebody $H_g$.

\begin{lemma}
\label{lemma:annular-twist}
Let $d,d'$ be a bounding pair that bounds an annulus in a handlebody $H_g$.  Then the bounding pair map $T_d T^{-1}_{d'}$ extends across the handlebody.
\end{lemma}

\begin{proof}
Cut $H_g$ along the annulus and reglue via a Dehn twist.  This gives a homeomorphism of the handlebody to itself that restricts to the bounding pair map on boundary surface.
\end{proof}

\begin{lemma}
\label{lemma:3-chain-twist}
Let $(a,b,c)$ be a 3-chain and suppose that one of the three curves bounds in the handlebody $H_g$.  Then the corresponding boundary pair $T_d T^{-1}_{d'}$ extends across the handlebody.
\end{lemma}

\begin{proof}
Without loss of generality, assume that either $a$ or $b$ bound in the handlebody.  Then the separating curve $s$, which is the boundary of a neighborhood of the 2-chain $(a,b)$, bounds in the handlebody.  Consequently, the curve $d'$ is obtained from $d$ by a band sum with this curve $s$.  It is now clear that $d$ and $d'$ bound an annulus in the handlebody.  Thus by Lemma \ref{lemma:annular-twist}, the bounding pair $T_d T^{-1}_{d'}$ extends across the handlebody.
\end{proof}

\section{Heegaard splittings and Trisections}

Throughout this section, let $H_g$ denote a 3-dimensional, genus $g$ handlebody and let $\Sigma_g$ be an abstract, closed surface of genus $g$.  Let $\cM(\del H_g)$ denote the mapping class group of $\del H_g$ and let $\cM(\Sigma_g)$ denote the mapping class group of $\Sigma_g$.  Let $\cH_g$ denote the subgroup of $\cM(\del H_g)$ consisting of classes that can be represented by a homeomorphism that extends across $H_g$.

A {\it cut system of disks} for $H_g$ is a collection of $g$ disjoint, properly embedded disks $\cD = \{D_i\}$ in $H_g$ whose complement is homeomorphic to $B^3$.    The boundaries of the disks are a collection $\alphas = \{\alpha_1,\dots,\alpha_g\}$ of disjoint, simple closed curves in $\del H_g$ that generate a $g$-dimensional subspace in $H_1(\del H_g;\ZZ)$.  If $\phi \in \cH_g$, then the cut system $\phi(\cD)$ can be connected to $\cD$ by a sequence of handleslides.  A {\it handleslide} of $D_i$ over $D_j$ consists of replacing $D_i$ by $D'_i$, which is obtained by joining $D_i$ to a second copy of $D_j$ by an embedded arc in $\del H_g$.

A collection $\alphas$ of $g$ disjoint, simple closed curves in $\Sigma_g$ whose complement has genus 0 is known as a {\it cut system of curves}.  A homeomorphism $\phi: \del H_g \rightarrow \Sigma_g$ sends the boundaries of a cut system of disks in $H_g$ to a cut system of curves $\alphas$.

We assume, once and for all, that we have fixed a cut system of disks $\cD$ with boundary $\alphas$ for $H_g$.

\subsection{Heegaard splittings}

Let $\Theta_g$ denote the set of pair $\theta = \{\theta_1,\theta_2\}$ where each $\theta_{\lambda}: \del H_g \rightarrow \Sigma$ is an isotopy class of homeomorphisms.  Given any element $\theta$, we can build a closed 3-manifold
\[Y_{\theta} = H_g \cup_{\theta_1\theta_2^{-1}} (-H_g)\]
This decomposition of $Y_{\theta}$ into two handlebodies is a {\it Heegaard decomposition}.

A collection of $2g$ simple closed curves $\{x_i,y_i\}$ on $\Sigma$ is a {\it geometric symplectic basis} if the geometric intersection numbers satisfy
\begin{align*}
i(x_i,y_j) &= \delta_{i,j} & i(x_i,x_j) &= 0 & i(y_i,y_j) &= 0.
\end{align*}

\begin{definition}
A Heegaard decomposition $S^3 = Y_{\theta}$, with $\theta = \{\theta_1,\theta_2\}$, is  {\it standard} if $\{\theta_1(\alpha_i),\theta_2(\alpha_i)\}$ is a geometric symplectic basis for $\Sigma_g$.
\end{definition}

\begin{definition}
A Heegaard decomposition $\#_k S^1 \times S^2 = Y_{\theta}$, with $\theta = \{\theta_1,\theta_2\}$, is  {\it standard} if there exists a geometric symplectic basis $\{x_i,y_i\}$ such that
\[\theta_1(\alpha_i) = x_i \qquad \theta_2(\alpha_i) = \begin{cases} x_i & \text{if } 1 \leq i \leq k \\ y_i & \text{if } k + 1 \leq i \leq g \end{cases}\]
\end{definition}

Many pairs $\{\theta_1,\theta_2\}$ specify the same Heegaard decomposition up to homeomorphism.  Recall that $\cH_g$ denotes the subgroup of $\cM(\del H_g)$ consisting of the classes of homeomorphisms that extend across the handlebody $H_g$.  Then $\cH_g \times \cH_g$ acts on $\Theta_g$ on the left as
\[ (\mu_1,\mu_2) \cdot \{\theta_1,\theta_2\} = \{\mu_1 \theta_1,\mu_2 \theta_2\}\]
and the mapping class group $\cM(\Sigma)$ acts on $\Theta_g$ on the right as
\[ \{\theta_1,\theta_2\} \cdot \rho = \{\theta_1 \rho, \theta_2 \rho\}\]
Let $\cG_2 = \cH_g \times \cH_g \times \cM(\Sigma)^{op}$.

\begin{lemma}
The $\cG_2$-orbits of $\Theta_g$ are precisely the equivalence classes of Heegaard splittings.
\end{lemma}

An important and well-known fact, due to Waldhausen, is that the standard Heegaard splitting of $S^3$ is essentially unique.

\begin{theorem}[\cite{Waldhausen}]
\label{thrm:Waldhausen}
Suppose that the pair $\theta = \{\theta_1,\theta_2\}$ determines a Heegaard splitting of $S^3$.  Then
\[\{\theta_1,\theta_2\} \sim \{\iota_1,\iota_2\}\]
where $\iota = \{\iota_1,\iota_2\}$ is standard.
\end{theorem}

By inductively applying Haken's Lemma, a similar statement holds for connected sums of $S^1 \times S^2$.

\begin{corollary}
\label{cor:Waldhausen}
Suppose that the pair $\theta = \{\theta_1,\theta_2\}$ determines a Heegaard splitting of $\#_k S^1 \times S^2$.  Then
\[\{\theta_1,\theta_2\} \sim \{\iota_1,\iota_2\}\]
where $\iota = \{\iota_1,\iota_2\}$ is standard.
\end{corollary}

The homology of $Y_{\theta}$ can be computed using the following chain complex
\[ \xymatrix{
0 \ar[r] & \ZZ \ar[r]^0 & \ZZ^g \ar[r]^Q & \ZZ^g \ar[r]^0 & \ZZ \ar[r] & 0
}\]
All of the maps are nonzero except for $Q$.  Let $\{x_i\}$ denote a basis for $\ZZ^g$ and $\{y_j\}$ denote a second basis.  The linear map $Q$ is defined by the formula
\[Q y_j = \sum_{i = 1}^g \langle \theta_1(\alpha_i), \theta_2(\alpha_j)\rangle_{\Sigma}\]
where $\langle,\rangle_{\Sigma}$ denotes the intersection pairing on $H_1(\Sigma;\ZZ)$.  

\pagebreak

\begin{proposition}
\label{prop:heegaard-ZHS}
Let $\theta = \{\theta_1,\theta_2\} \in \Theta_g$ and $Y_{\theta}$ the associated 3-manifold.
\begin{enumerate}
\item $Y_{\theta}$ is an integral homology 3-sphere if and only if $Q$ is unimodular.
\item If $\rho \in \cI_g$ and $\iota = \{\iota_1,\iota_2\}$ is a standard Heegaard decomposition of $S^3$, then the pair $\{\iota_1,\iota_2 \rho\}$ determines an integral homology 3-sphere $Y_{\rho}$.
\item If $Y_{\theta}$ is an integral homology, then there exists some standard Heegaard decomposition of $S^3$ $\iota = \{\iota_1,\iota_2\}$ and an element $\rho \in \cI_g$ such that
\[\{ \theta_1,\theta_2\} \sim \{\iota_1,\iota_2 \rho\}.\]
\end{enumerate}
\end{proposition}

Morita proved a stronger version of Proposition \ref{prop:heegaard-ZHS}, which we state in the current formalism.

\begin{theorem}[\cite{Morita-1}]
\label{thrm:Morita}
If $Y_{\theta}$ is an integral homology 3-sphere, then there exists some standard Heegaard decomposition of $S^3$ $\iota = \{\iota_1,\iota_2\}$ and an element $\rho \in \cK_g$ such that
\[\{ \theta_1,\theta_2\} \sim \{\iota_1,\iota_2 \rho\}.\]
\end{theorem}

\subsection{Trisections of closed, smooth 4-manifolds}

Recall that a $(g;k_1,k_2,k_3)$ trisection $\cT_X$ of a closed, oriented, smooth 4-manifold is a decomposition $X = Z_1 \cup Z_2 \cup Z_3$, where each double intersection $H_{\lambda} = Z_{\lambda - 1} \cap Z_{\lambda}$ is a genus $g$ handlebody and $\Sigma = Z_1 \cap Z_2 \cap Z_3$ is a closed, oriented genus $g$ surface.  We can identify each double intersection with a fixed, abstract $H_g$ and the central surface with a fixed, abstract $\Sigma_g$.  These identification induce a triple of maps
\[\phi_{\lambda}: \del H_g \rightarrow \Sigma_g\]
for $\lambda = 1,2,3$.  Thus, every trisection $\cT_X$ determines a triple $\phi = \{\phi_1,\phi_2,\phi_3\}$ of homeomorphisms.  This triple uniquely describes $X$ up to diffeomorphism.  Furthermore, the pair $\{\phi_{\lambda},\phi_{\lambda+1}\}$ determines a Heegaard splitting of $\#_{k_{\lambda}} S^1 \times S^2$.  

The triple $\phi$ depends on choices.  For any choices $\mu_{\lambda} \in \cH_g$ and $\rho \in \cM(\Sigma_g)$, the triple $\{\mu_1 \phi_1 \rho, \mu_2 \phi_2 \rho, \mu_3 \phi_3 \rho\}$  determines a diffeomorphic 4-manifold.  Let $\Phi_g$ denote the set of triples $\phi = \{\phi_1,\phi_2,\phi_3\}$ where each $\phi_{\lambda}: \del H_g \rightarrow \Sigma$ is an isotopy class of homeomorphisms.  The group $\cH_g \times \cH_g \times \cH_g$ acts on $\Phi_g$ on the left by the rule
\[ (\mu_1,\mu_2,\mu_3) \cdot \{\phi_1,\phi_2,\phi_3\} = \{\mu_1 \phi_1, \mu_2 \phi_2, \mu_3 \phi_3\}\]
and the group $\cM(\Sigma)$ acts on $\Phi_g$ on the right by the rule
\[ \{\phi_1,\phi_2,\phi_3\} \cdot \rho = \{\phi_1 \rho,\phi_2 \rho, \phi_3 \rho\}\]
Let $\cG_{3} = \cH_g \times \cH_g \times \cH_g \times \cM(\Sigma)^{op}$.  We can combine the above actions into a single $\cG_3$ action on $\Phi_g$ and each diffeomorphism class of trisection corresponds with a unique $\cG_3$-orbits.  We refer to the action of $\cM(\Sigma)$ as {\it global reparametrization} and the action of $\cH_g \times \cH_g \times \cH_g$ as a {\it handlebody diffeomorphism}.

\begin{remark}
If $\phi = \{\phi_1,\phi_2,\phi_3\}$ is a triple arising from a trisection of a closed 4-manifold, then by Theorem \ref{thrm:Waldhausen} and Corollary \ref{cor:Waldhausen}, each triple $\{\phi_{\lambda},\phi_{\lambda+1}\}$ is $\cG_2$-equivalent to a standard $\{\iota_1,\iota_2\}$.  However, it is not true in general that $\phi$ is $\cG_3$-equivalent to some $\{\iota_1,\iota_2,\iota_3\}$ where every pair $\{\iota_{\lambda},\iota_{\lambda+1}\}$.  It is known that only connected sums of $\CP^2, \overline{\CP}^2$ and $S^2 \times S^2$ admit this property.
\end{remark}

It is often convenient to work just in $\cM(\Sigma_g)$ and interpret $\cH_g$ as a subgroup of $\cM(\Sigma_g)$.  Let $\phi =  \{\phi_{\lambda}\}$ be a fixed triple.  Define subgroups of $\cM(\Sigma)$
\begin{align*}
\cA_{\phi} & \coloneqq \{ \phi^{-1}_{1} \mu \phi_1 : \mu \in \cH_g \} \\
\cB_{\phi} & \coloneqq \{ \phi^{-1}_{2} \mu \phi_3 : \mu \in \cH_g \} \\
\cC_{\phi} & \coloneqq \{ \phi^{-1}_{3} \mu \phi_3 : \mu \in \cH_g \} \\
\end{align*}
Consequently, the orbits of the action of $\cH_g \times \{1\} \times \{1\}$ on the left are precisely the orbits of the action of $\cA$ on the right.  Similarly for the orbits of $\cB$ and $\cC$.  Furthermore, let $\cA\cB_{\phi} = \cA_{\phi} \cap \cB_{\phi}$.  

\begin{lemma}
\label{lemma:AB-equiv}
Fix a triple $\phi = \{\phi_1,\phi_2,\phi_3\}$.  If $\rho \in \cA\cB_{\phi}$, then 
\[ \{\phi_1,\phi_2,\phi_3\} \sim \{\phi_1,\phi_2,\phi_3 \rho\}\]
\end{lemma}

\begin{proof}
Since $\rho \in \cA\cB_{\phi} = \cA_{\phi} \cap \cB_{\phi}$, there exist $a,b \in \cH_g$ such that
\[ a \phi_1 = \phi_1 \rho \qquad b \phi_2 = \phi_2 \rho\]
Thus
\[ \{\phi_1,\phi_2,\phi_3\} \sim \{\phi_1 \rho,\phi_2 \rho,\phi_3 \rho\} \sim \{a \phi_1,b\phi_2,\phi_3 \rho\} \sim \{\phi_1,\phi_2,\phi_3 \rho \}\]
\end{proof}

\subsection{Heegaard diagrams}  For many purposes, such as computing the algebraic topology of $X$, it is often easier and sufficient to encode $\phi_i : \del H_g \rightarrow \Sigma_g$ by its action on a cut system. 

Given a collection of maps $\{ \phi_{\lambda}\}$ that determine a trisection of a 4-manifold $X$, we obtain a {\it trisection diagram} $(\Sigma, \alphas_1,\alphas_2,\alphas_3)$, where $\alpha_{\lambda} = \phi_{\lambda}(\alphas)$.  To agree with conventions in Heegaard Floer theory, we often denote the trisection diagram as $(\Sigma,\alphas,\betas,\gammas)$ instead.  A collection of the form $(\Sigma,\alphas,\betas,\gammas)$, where each of $\alphas,\betas$ and $\gammas$ is a cut system of curves, is known as a {\it Heegaard triple}.  Conversely, tiven a Heegaard triple $(\Sigma,\alphas_1,\alphas_2,\alphas_3)$, we can choose a (nonunique) triple of homeomorphisms $\{\phi_{\lambda}: \del H_g \rightarrow \Sigma_g\}$ such that $\alphas_{\lambda}$ is the image under $\phi_{\lambda}$ of a fixed cut system of curves.  

\begin{lemma}
A Heegaard triple $(\Sigma,\alphas,\betas,\gammas)$ determines a unique $\cG_3$-orbit in $\Phi_g$.
\end{lemma}

The cohomology ring of $X$ can be computed purely from the Heegaard triple $(\Sigma,\alphas,\betas,\gammas)$.  In particular, the cut systems $\alphas,\betas,\gammas$ determines a $g$-dimensional subspace $L_{\alpha},L_{\beta},L_{\gamma}$ in $H_1(\Sigma;\ZZ)$.  As detailed in \cite{FKSZ}, the cohomology ring $H^*(X;\ZZ)$ of $X$ is determined purely by the triple $L_{\alpha},L_{\beta},L_{\gamma}$, up to a symplectic automorphism of $H_1(\Sigma;\ZZ)$.

Let $(\Sigma,\alphas,\betas,\gammas)$ be a trisection diagram for a $(g;0,k,0)$-trisection.  Let $_{\alpha} Q_{\beta}$ be the $g \times g$-matrix of intersection numbers $\langle \alpha_i,\beta_j \rangle_{\Sigma}$ and define $_{\beta} Q_{\gamma}$ and $_{\gamma}  Q_{\alpha}$ similarly.  Then the intersection form $Q_X$ is given in matrix-form by the formula (\cite[Theorem 4.3]{FKSZ})
\begin{equation}
\label{eq:q-formula}
Q_X = (-1) \cdot _{\gamma} Q_{\beta} \cdot (_{\alpha} Q_{\beta})^{-1}  \cdot _{\gamma} Q_{\alpha}
\end{equation}

\begin{theorem}[\cite{FKSZ}]
\label{thrm:q-formula}
Let $\cT_X$ be a $(g;0,k,0)$-trisection of $X$.  Then $\cT_X$ admits a diagram $(\Sigma,\alphas,\betas,\gammas)$ such that
\begin{enumerate}
\item $(\Sigma,\alpha,\betas)$ is a standard Heegaard diagram for $S^3$.
\item In $H_1(\Sigma)$, we have
\[ [\gamma_i] = - [\alpha_i] - \sum_{j = 1}^g \widetilde{Q}_{i,j} [\beta_j]\]
where $Q$ is the intersection form of $X$ and $\widetilde{Q} = Q \oplus \langle 0 \rangle^k$.
\end{enumerate}
\end{theorem}

Combining the above theorem with Proposition \ref{prop:cut-system-torelli}, we obtain

\begin{corollary}
\label{cor:trisection-torelli-equiv}
Let $X,Y$ be closed, smooth, oriented 4-manifolds such that $Q_X \cong Q_y$.  Suppose that $X,Y$ admit $(g;0,k,0)$-trisections $\cT_X$ and $\cT_Y$.  Then there exists a diagram $(\Sigma,\alphas,\betas,\gammas)$ for $\cT_X$ and an element $\rho \in \cI_g$ such that $(\Sigma,\alphas,\betas,\rho(\gammas))$ is a diagram for $\cT_Y$.
\end{corollary}

\subsection{Pseudotrisections}

We introduce the more general notion of a {\it pseudotrisection}.  Recall from the previous subsection that a trisection of a closed, smooth, oriented 4-manifold can be encoded by a triple of maps $\{\phi_{\lambda}: \del H_g \rightarrow \Sigma_g\}$ or a Heegaard triple $(\Sigma,\alphas,\betas,\gammas)$.  In an honest trisection, the 3-manifolds obtained as the union of any pair of handlebodies must by homeomorphic to $S^3$ or $\#_k S^1 \times S^2$ for some $k \geq 1$.  In a pseudotrisection, we relax the condition and allow two of the 3-manifolds to merely have the same $\ZZ$-homology as one of those manifolds.  

\begin{definition}
Let $(\Sigma,\alphas,\betas,\gammas)$ be a diagram that defines a $(g;0,k,0)$-pseudotrisection $p\cT$.  The {\it intersection form} $Q_{p\cT}$ of the pseudotrisection is the quadratic form determines by the formula in Equation \ref{eq:q-formula}.
\end{definition}

For each unimodular intersection form $Q$, we can build a pseudotrisection with intersection form $Q$ as follows:

\begin{proposition}
\label{prop:pseudo-Q}
Let $Q$ be a unimodular, symmetric bilinear form of rank $g$.  Then there exists a $(g+k;k,0,0)$ pseudotrisection $(\Sigma,\alphas_Q,\betas_Q,\gammas_Q)$ such that
\begin{enumerate}
\item $\alpha,\beta$ are standard,
\item $\beta,\gamma$ are standard, and
\item $[\gamma_i] = -[\alpha_i] - \sum_{j = 1}^g \widetilde{Q}_{i,j} [\beta_j]$
\end{enumerate}
where $\widetilde{Q} = Q \oplus \langle 0 \rangle^k$.
\end{proposition}

\begin{proof}

Start with the standard $(g,0)$ trisection $(\Sigma,\alphas,\betas,\gammas')$ of $\#_g \CP^2$.  The standard $(1,0)$-trisection of $\CP^2$ is given by the diagram $(\Sigma,\alpha,\beta,\gamma)$ where $\alpha,\beta$ intersect geometrically once and $\gamma$ represents the class $-[\alpha] - [\beta]$.  Taking the connected sum of $g$ copies gives the standard $(g,0)$ trisection of $\#_g \CP^2$.

Let $b_i, b_j$ be curves homotopic to $\beta_i$ and $\beta_j$, respectively, and disjoint from $\betas$.  In addition, let $b_{i,j}$ be a band sum of $b_i$ and $b_j$ disjoint from $\betas$.  Now set
\[\phi = \left(\prod_{i = 1}^g T_{b_i}^{Q_{i,i} - 1}\middle)\middle(\prod_{1 \leq i < j \leq g} (T_{b_{i,j}}T^{-1}_{b_i}T^{-1}_{b_j})^{Q_{i,j}}\right)\]
and let $\gammas = \phi(\gammas')$.  We can assume that each Dehn twist is along a curve disjoint from a fixed collection of curves $\{\beta_1,\dots,\beta_g\}$.  Thus, the geometric intersection number between $\betas$ and $\gammas$ never changes.  In particular, $\gammas$ and $\betas$ still form a geometric symplectic basis for the surface.  The induced map $\phi_*$ on $H_1(\Sigma_g)$ satisfies
\[\phi_*([\alpha_i] + [\beta_i]) = [\alpha_i] + \sum_{j = 1}^g Q_{i,j}[\beta_j].\]
Therefore, $(\Sigma,\alphas,\betas,\gammas)$ gives a $(g;0)$ pseudotrisection with intersection form $Q$.

To obtain a $(g+k;k,0,0)$ pseudotrisection, we can stabilize $k$ times by connected sum with the standard $(1;1,0,0)$ trisection of $S^4$.  A diagram for this trisection is $(T^2,\alpha,\alpha,\beta)$ where $(T^2,\alpha,\beta)$ is a standard Heegaard diagram for the genus 1 splitting of $S^3$.
\end{proof}

We refer to the pseudotrisection $\cT_Q = (\Sigma,\alphas_Q,\betas_Q,\gammas_Q)$ constructed in Proposition \ref{prop:pseudo-Q} as the {\it standard pseudotrisection} for $Q$.

\pagebreak

\begin{proposition}
\label{prop:pseudo-tri-standardize}
Let $Q$ be an intersection form.
\begin{enumerate}
\item Let $\cT = (\Sigma,\alphas',\betas',\gammas')$ a $(g;k,0,0)$-pseudotrisection with intersection form $Q$.  There exists an equivalent diagram $(\Sigma,\alphas,\betas,\gammas)$ for $\cT$ such that $\{[\alpha_i],[\beta_j]\}$ is a symplectic basis for $H_1(\Sigma)$ and 
\[ [\gamma_i] = - [\alpha_i] - \sum_{j = 1}^g \widetilde{Q}_{i,j} [\beta_j]\]
where $\widetilde{Q} = Q \oplus \langle 0 \rangle^k$.
\item  Let $\cT = (\Sigma,\alphas',\betas',\gammas')$ a $(g;k,0,0)$-pseudotrisection diagram satisfying the conclusions of part (1).  Then there exists some $\phi \in \cI_g$ such that $\cT_{Q,\phi} \coloneqq (\Sigma,\alphas_Q,\betas_Q,\phi(\gammas_Q))$ is equivalent to $\cT$.
\item Let $\cT_1 = (\Sigma_g,\alphas_1,\betas_1,\gammas_1)$ and $\cT_2 = (\Sigma_g,\alphas_2,\betas_2,\gammas_2)$ be $(g+k;k,0,0)$-pseudotrisections that satisfy the conclusions of part (1) and have the same intersection form $Q$.  Then there exists some $\phi \in \cI_g$ such that $\phi(\gammas_1) = \gammas_2$.
\end{enumerate}
\end{proposition}

\begin{proof}
The proof of (1) is essentially the same as the proof of Theorem \ref{thrm:q-formula}, which is a restatement of \cite[Theorem 4.4]{FKSZ}.  The only difference is that we cannot assume $(\Sigma,\alphas,\betas)$ is geometrically standard.  However, the proof only requires that $\{[\alpha_i],[\beta_i]\}$ form a symplectic basis for $H_1(\Sigma)$.  This follows since the 3-manifold constructed from the diagram $(\Sigma,\alphas,\betas)$ is an integral homology 3-sphere.

Part (2) follows from part (1) by Proposition \ref{prop:cut-system-torelli}.

To prove (3), note that by part (2) we can assume $\cT_1$ has a diagram $(\Sigma,\alphas_Q,\betas_Q,\phi_1(\gammas_Q))$ and $\cT_2$ has a diagram $(\Sigma,\alphas_Q,\betas_Q,\phi_2(\gammas_Q))$.  Now set $\phi = \phi_2 \circ \phi_1^{-1}$.
\end{proof}

\section{Trisections and Torelli group}

The main result of this section is an analogue of Morita's result (Theorem \ref{thrm:Morita}) for  trisections of closed, smooth 4-manifolds.

\begin{theorem}
\label{thrm:torelli-trisection}
Let $Q$ be a unimodular, symmetric bilinear form of rank $n$ over $\ZZ$.  Let $X$ be a closed, smooth 4-manifold with intersection form $Q$ and which admits a $(g;k,0,0)$ trisection $\cT_X$.  Let $(\Sigma,\alphas_Q,\betas_Q,\gammas_Q)$ be the pseudotrisection from Part (1).  Then there exists a map $\phi \in \cK_{g}$ such that $(\alphas_Q,\betas_Q,\phi(\gammas_Q))$ is a trisection diagram for the trisection $\cT_X$.
\end{theorem}

Fix an intersection form $Q$ and let $\cT_Q$ be the standard $(g;k,0,0)$-pseudotrisection with intersection form $Q$.  For concreteness, let $\phi = \{\phi_1,\phi_2,\phi_3\}$ be a triple of maps $\phi_{\lambda}: \del H_g \rightarrow \Sigma_g$ encoding this pseudotrisection.  We then obtain subgroups $\cA_{\phi}, \cB_{\phi}, \cC_{\phi} \subset \cM(\Sigma_g)$ of homeomorphisms that extend across the $\alpha-\beta-$ and $\gamma$-handlebodies, respectively.  From now own, we supress the subscript $\phi$, although it is important that these groups depend on $\phi$.

Let $\cT \cA$ be the intersection of $\cA$ with the Torelli group $\cI_g$.  Define $\cT \cB, \cT \cC$ and $\cT \cA \cB$ similarly, where $\cA \cB$ denotes the intersection of $\cA$ and $\cB$.  Finally, let $\tau: \cI_{g,1} \rightarrow \Lambda^3 H_1(\Sigma_g)$ denote the Johnson homomorphism.  

\begin{proposition}
\label{prop:Johnson-surjective}
Let $Q$ be a unimodular, symmetric bilinear form and let $(\Sigma,\alphas_Q,\betas_Q,\gammas_Q)$ be the standard $(g;k,0,0)$-pseudotrisection for $Q$.  Then
\[ \tau(\cT \cA \cB) + \tau(\cT \cC) = \Lambda^3 H_1(\Sigma) \]
\end{proposition}

\pagebreak

\begin{proof}
We just need to find explicit bounding pairs that map onto a basis for $\Lambda^3 H_1(\Sigma_g)$.   

To do this, we decompose $\Lambda^3 H_1(S_g)$ into subspaces and then check that each subspace is in $\tau(\cT \cA \cB) + \tau(\cT \cC)$.  Let $\{x_i,y_i\}$ for $1 \leq i \leq g - k$ be the {\it left} generators and $\{x_l,y_l\}$ for $g - k + 1 \leq l \leq g$ be the {\it right} generators.

Let $W^1_X$ be the subspace spanned by elements of the form $x_i \wedge a \wedge b$, where $x_i,a,b$ are left generators; let $W^1_Y$ be the subspace spanned by elements of the form $y_i \wedge a \wedge b$, where $y_i,a,b$ are left generators; let $W^1_{XY} = W^1_X \cap W^1_Y$ be their intersection and let $W^1 = W^1_X + W^1_Y$ their sum.  Define $W^2_X,W^2_Y,W^2_{XY},W^2$ similarly using only right generators.  Finally, define $W_X, W_Y, W_{XY},W$ in the same manner but allowing elements of any combination between left and right.

By assumption, we know that $\{x_1,\dots,x_g\}$ bound in $H_{\alpha}$, that $\{y_1,\dots,y_{g-k},x_{g-k+1},\dots,x_g\}$ bound in $H_{\beta}$ and some curves $\{z_1,\dots,z_g\}$ bound in $H_{\gamma}$, where
\[ [z_i] =
\begin{cases}
[x_i] + Q[y_i] & \text{for } 1 \leq i \leq g- k \\
[y_i] & \text{for } g-k+1 \leq i \leq g
\end{cases} \]

First, it follows from Morita's original argument that $W^2$ is in the image $\tau(\cT\cA \cB) + \tau(\cT\cC)$.  Specifically, the following elements lie in $\tau(\cT \cA \cB) \cap \tau (\cT \cC)$:
\begin{align*}
[x_i] \wedge [y_i] \wedge ([x_i] + [x_j]) && [y_i] \wedge [x_i] \wedge ([y_i] + [y_j]) \\
 [x_i] \wedge ([y_i] + [x_j]) \wedge ([y_j] + [y_k]) && [y_i] \wedge ([x_i] + [y_j]) \wedge ([x_j] + [x_k]) \\
\end{align*}
In addition, the following elements
\[ [x_i] \wedge ([y_i] + [x_j]) \wedge ([y_j] + [x_k])  \] 
lie in $\tau( \cT \cA \cB)$ and the elements
\[ [y_i] \wedge ([x_i] + [y_j]) \wedge ([x_j] + [y_k]) \]
 lie in $\tau(\cT \cC)$.  These elements suffice to span $W^2$.

Next, let $1 \leq i,j,l \leq g-k$.  The following elements are each the image of a 3-chain, at least one of which bounds in $H_{\alpha}$ and at least one of which bounds in $H_{\beta}$.  Thus, by Lemma \ref{lemma:3-chain-twist}, they lie in the image $\tau(\cA\cB)$:
\begin{align*}
[x_i] \wedge [y_i] \wedge ([x_i] + [x_j]) && [y_i] \wedge [x_i] \wedge ([y_i] + [y_j]) \\
([x_i] + [x_l]) \wedge [y_i] \wedge ([x_i] + [x_j]) &&  ([y_i]+ [y_l]) \wedge [x_i] \wedge ([y_i] + [y_j])
\end{align*}
Thus $W^1_{XY} \subset \tau(\cT\cA\cB)$.

Next, since $\{y_i,z_i\}$ are geometrically standard, we can find 3-chains that live in $\cC$ mapping to
\begin{align*}
[y_i] \wedge ([x_i] + Q[y_i]) \wedge ([y_i] + [y_l]) &= [y_i] \wedge Q[y_i] \wedge [y_j] && \text{ mod } W_{XY} \\
([y_i] + [y_l]) \wedge ([x_i] + Q[y_i]) \wedge ([y_i] + [y_l]) &= ([y_i] + [y_l]) \wedge Q[y_i] \wedge ([y_i] + [y_l]) && \text{ mod } W_{XY}
\end{align*}
Since $Q$ is unimodular, we can express any $[y_m]$ as a $\ZZ$-linear combination of the $\{Q[y_i]\}$ and therefore $W^1_Y \subset \tau(\cA\cB) + \tau(\cC)$.  Finally, it again follows from Morita's original proof that
\[ [z_i] \wedge [z_j] \wedge [z_i] = [x_i] \wedge [x_j] \wedge [x_k] \text{ mod }W_Y\]
is in $\tau(\cC)$.  Consequently, $W^2 \subset \tau(\cT\cA\cB) + \tau(\cT\cC)$.

Finally, we need to ensure all mixed elements lie in the image.  From now on, we will let $1 \leq i,j \leq g - k$ and $g-k+1 \leq l,m \leq g$.

By a similar argument as above, we can obtain the elements
\begin{align*}
[x_i] \wedge [y_i] \wedge [x_l] && ([x_i] + [x_j]) \wedge [y_i] \wedge [x_l] && [x_i] \wedge ([y_i] + [y_j]) \wedge [x_l] \\
[x_i] \wedge [y_i] \wedge [y_l] && ([x_i] + [x_j]) \wedge [y_i] \wedge [y_l] && [x_i] \wedge ([y_i] + [y_j]) \wedge [y_l] \\
[x_l] \wedge [y_l] \wedge [x_i]  && ([x_l] + [x_m]) \wedge [y_l] \wedge [x_i] && [x_l] \wedge ([y_l] + [y_m]) \wedge [x_i] \\
[x_l] \wedge [y_l] \wedge [y_i] && ([x_l] + [x_m]) \wedge [y_l] \wedge [y_i] && [x_l] \wedge ([y_l] + [y_m]) \wedge [y_i]
\end{align*}
in the image $\tau(\cT\cA\cB)$.  This is enough to completely span $W_{XY}$.

Moreover, the elements
\begin{align*}
[x_i] \wedge [y_i + x_l] \wedge [x_j] &= [x_i] \wedge [x_l] \wedge [x_j] && \text{mod } W_{XY}\\
[x_i] \wedge [y_i + x_l] \wedge [x_m] &= [x_i] \wedge [x_l] \wedge [x_m] && \text{mod } W_{XY}
\end{align*}
also lie in $\tau(\cT\cA \cB)$ and so we can span $W_X$.  And finally, the elements
\begin{align*}
[y_i] \wedge ([z_i] + [z_l]) \wedge [y_j] &= [y_i] \wedge ([y_i] + [y_l]) \wedge [y_j] && \text{mod } W_{X} \\
[y_i] \wedge ([z_i] + [z_l]) \wedge [y_m] &= [y_i] \wedge ([y_i] + [y_l]) \wedge [y_m] && \text{mod } W_{X}
\end{align*}
are in the image $\tau(\cT\cC)$.  We have thus constructed a basis for $\Lambda^3 H_1$ that lies in $\tau(\cT\cA\cB) + \tau(\cT\cC)$.

\end{proof}

\begin{theorem}
Let $X$ be a closed, oriented, smooth 4-manifold $X$ that admits a $(g;k,0,0)$ trisection.  Then $X$ admits a trisection diagram
\[\cT_{Q,\phi} = (\Sigma,\alphas_Q,\betas_Q,\phi(\gammas_Q))\]
where $Q$ is the intersection form of $X$, $\cT_Q = (\Sigma,\alphas_Q,\betas_Q,\gammas_Q)$ is the standard pseudotrisection for $Q$ constructed in Proposition \ref{prop:pseudo-Q}, and $\phi \in \cK_g$.
\end{theorem}

\begin{proof}
Just as in the proof of Corollary \ref{cor:trisection-torelli-equiv}, we can apply Proposition \ref{prop:cut-system-torelli} and find a trisection diagram $(\Sigma,\alphas_Q,\betas_Q,\psi(\gammas_Q))$ for $X$, where $\psi$ is an element of the Torelli group $\cI_g$.  By Proposition \ref{prop:Johnson-surjective}, we can find some elements $a \in \cT \cA \cB$ and $c \in \cT \cC$ such that 
\[\tau(\psi) = \tau(a) + \tau(c)\]
Consequently, $\tau(c^{-1} \psi a^{-1}) = -\tau(c) + \tau(\phi) - \tau(a) = 0$.  Thus, $\phi = c^{-1} \psi a^{-1} \in \text{ker}(\tau) = \cK_g$.  It now follows from Lemma \ref{lemma:AB-equiv} that the trisection diagrams
\[(\Sigma,\alphas_Q,\betas_Q,\phi(\gammas_Q)) \qquad \text{and} \qquad (\Sigma,\alphas_Q,\betas_Q,\psi(\gammas_Q))\]
are equivalent.
\end{proof}

\section{Casson invariant and linking forms}

\subsection{Casson invariant}

If $K$ is a knot in a homology 3-sphere $Y$, let $Y + \frac{1}{m} K$ denote the integral homology 3-sphere obtained by Dehn surgery on $K$ with slope $\frac{1}{m}$. 

Let $\lambda(Y) \in \ZZ$ denote the Casson invariant of an integral homology 3-sphere $Y$.  It has the following properties:

\begin{enumerate}
\item $\lambda(S^3) = 0$.
\item For any $Y$, any knot $K$ in $Y$ and any integer $m \in \ZZ$, the difference
\[\lambda'(K) \coloneqq \lambda \left(Y + \frac{1}{m+1} K \right) - \lambda \left(Y + \frac{1}{m} K \right)\]
is independent of $m$.
\item If $K \cup L$ is a boundary link in $Y$, then
\begin{align*}
\lambda''(K,L) &\coloneqq \lambda\left(Y + \frac{1}{m+1}K + \frac{1}{n+1}L\right) - \lambda\left(Y + \frac{1}{m}K + \frac{1}{n+1}L\right)\\
& - \lambda\left(Y + \frac{1}{m+1}K + \frac{1}{n}L\right) + \lambda\left(Y + \frac{1}{m}K + \frac{1}{n}L\right) \\
&=0.
\end{align*}
\item $\lambda'(T(2,3)) = 1$.
\item $\lambda'(K) = \frac{1}{2} \Delta''_K(1)$.
\item $\lambda(-Y) = - \lambda(Y)$.
\item $\lambda(Y_1 \# Y_2) = \lambda(Y_1) + \lambda(Y_2)$.
\end{enumerate}
Moreover, we have that
\[\frac{1}{2} \Delta''_{K}(1) = \text{Arf}(K) \text{ mod }2.\]

\subsection{Linking form}

The Arf invariant and Casson invariant of knots can be computed in terms of a linking form on the homology of a Seifert surface.

Let $Y$ be an oriented 3-manifold and $\Sigma$ an oriented, embedded surface in $Y$.  The linking form is a map
\[l : H_1(\Sigma) \times H_1(\Sigma) \rightarrow \ZZ\]
defined in terms of the pairwise linking of curves on $\Sigma$.  Specifically, let $a$ and $b$ be simple closed curves in $\Sigma$.  Then
\[l([a],[b]) = lk_Y(a,b^+)\]
where $b^+$ denotes a pushoff of $b$ in the positive normal direction to $\Sigma$.  The linking pairing is well-defined on homology, bilinear, and satisfies the symmetry relation
\begin{equation}
\label{eq:linking-symmetry}
l([b],[a]) = l([a],[b]) + \langle [a],[b] \rangle_{\Sigma}.
\end{equation}
Reducing mod 2, we obtain a map $q: H_1(\Sigma; \ZZ/2\ZZ) \rightarrow \ZZ/2 \ZZ$ defined by setting
\[q(a) = l(a,a)\]
It is a {\it quadratic enhancement} of the intersection pairing $\langle , \rangle_{\Sigma}$, meaning that it satisfies the relation
\begin{equation}
\label{eq:quad-enhancement}
q(x + y) = q(x) + q(y) + \langle x,y \rangle_{\Sigma} \text{ mod }2
\end{equation}

Now, suppose that $\Sigma$ is a Seifert surface for a knot $K \subset Y$, where $Y$ is an integral homology 3-sphere.  Let $\{a_i,b_i\}$ be a geometric symplectic basis for $\Sigma$.  The Casson invariant $\lambda'(K)$ of $K$ can be computed in terms of the linking form on $\Sigma$ according to the formula
\begin{equation}
\lambda'(K) = \sum_{i = 1}^{g} \left( l(a_i,a_i) l(b_i,b_i) - l(a_i,b_i)l(a_i,b_i) \right) + 2 \sum_{1 \leq i < j \leq g} \left( l(a_i,a_j)l(b_i,b_j) - l(a_i,b_j)l(a_j,b_i) \right)
\end{equation}
The Arf invariant can be computing using the simpler formula
\begin{equation}
\label{eq:Arf}
\text{Arf}(K) = \sum_{i = 1}^g q(a_i)\cdot q(b_i) \text{ mod } 2.
\end{equation}

\subsection{Linking forms on the central surface}

Let $\Sigma$ be the central surface of a pseudotrisection.  This surface embeds in each 3-manifold $Y_i$ as a Heegaard surface.  Each embedding determines a linking form $l_i$ on $H_1(\Sigma)$ and a quadratic enhancements $q_i$.  In general, these linking forms are distinct.  However, with respect to a particular basis, we can describe $l_2$ and $l_3$ in terms of the quadratic form of the pseudotrisection.  In addition, let $c$ be a separating simple closed curve in $\Sigma$.  Then, via the embedding of $\Sigma$ in $Y_i$, it determines a knot $K_i$.  In general, the Arf invariant, Alexander polynomial and Casson invariant of $K_i$ depends on $i$.  However, these also can be described in terms of the quadratic form of the pseudotrisection.

Throughout this subsection, let $(\Sigma,\alphas,\betas,\gammas)$ be a pseudotrisection with quadratic form $Q$.  Let $Y_2 = H_{\beta} \cup -H_{\gamma}$ and $Y_3 = H_{\gamma} \cup - H_{\alpha}$.

First, we describe the difference between the quadratic enhancements $q_2$ and $q_3$.

\begin{lemma}
\label{lemma:Arf-quadratic}
Suppose that $Q$ is even.  Then $q_2 = q_3$.
\end{lemma}

\begin{proof}
By Equation \ref{eq:quad-enhancement}, it is enought to check on the symplectic basis $\{[z_i],[x_i]\}$.  Furthermore, we assume that we have performed handeslides so that the diagram satisfies the conclusions of Proposition \ref{prop:pseudo-tri-standardize}.

First, the mod 2 linking form $q_3$ in $Y_{3}$ vanishes on each element of the basis $\{[x_i],[z_i]\}$ since we can choose representatives that bound disks in one of the two handlebodies $H_{\gamma}$ and $H_{\alpha}$.  By the same argument, the mod 2 linking form $q_2$ in $Y_{2}$ vanishes on the basis $\{[y_i],[z_i]\}$ since each can be represented by the boundary of a compressing disk.  To check that the mod 2 linking forms agree, we need to check that $q_2([x_i])$ always vanishes.  By Equation \ref{eq:quad-enhancement}, we have
\[q_2([x_i]) = q_2(-[z_i] + Q[y_i]) = q_2([z_i]) + q_2(Q[y_i]) + \langle [z_i], Q[y_i] \rangle_{\Sigma} = Q_{i,i} \, \, \text{ mod }2.\]
But since $Q$ is an even intersection form, every diagonal element $Q_{i,i}$ is even and vanishes mod $2$.
\end{proof}

As a corollary, we see that the Arf invariants of $K_2$ and $K_3$ agree when the intersection form is even.

\begin{lemma}
\label{lemma:same-Arf}
Suppose that $Q$ is even.  Let $c$ be a separating curve in $\Sigma$ and let $K_2$ and $K_3$ be the knots obtained as the image of $c$ in $Y_2$ and $Y_3$.  Then $\text{Arf}(K_2) = \text{Arf}(K_3)$.
\end{lemma}

\begin{proof}
Since $c$ is separating on $\Sigma$, it cuts off a subsurface $\Sigma'$.  The images of this surface in $Y_2$ and $Y_3$ are Seifert surfaces for $K_2$ and $K_3$, respectively.  Choose a symplectic basis $\{a_i,b_i\}$ for $H_1(\Sigma')$.  By Equation \ref{eq:Arf} and Lemma \ref{lemma:Arf-quadratic}, 
\begin{align*}
\text{Arf}(K_2) &= \sum q_2(a_i) q_2(b_i)\\
&= \sum q_3(a_i) q_3(b_i) \\
&= \text{Arf}(K_3).
\end{align*}
\end{proof}

Now we describe the linking forms $l_2$ and $l_3$ in terms of the quadratic form of the pseudotrisection.  Recall that we have a symplectic basis $\{x_i,y_i\}$ for $H_1(\Sigma,\ZZ)$ such that $\{x_i\}$ is a basis for $L_{\alpha}$ and $\{y_i\}$ is a basis for $L_{\beta}$  Furthermore.  To simplify notation, let $Qy_i = \sum_{j = 1}^n Q_{i,j} y_j$.  Then $z_i = - x_i - Qy_i$ is a basis for $L_{\gamma}$.

\begin{proposition}
\label{prop:linking-l2}
The linking form $l_2$ satisfies
\begin{align*}
l_2(z_i,z_j) &= 0 & l_2(z_i,x_j) &= -Q_{i,j}\\
l_2(x_i,z_j) &= 0 & l_2(x_i,x_j) &= Q_{i,j}\\
l_2(x_i,y_j) &= -\delta_{i,j} & l_2(y_i,y_j) &= 0
\end{align*}
\end{proposition}

\begin{proof}
For any $A \in H_1(\Sigma;\ZZ)$, we have that
\begin{align*}
l_2(y_i,A)& = 0 &  l_2(A,z_j) &= 0
\end{align*}
Consequently, we easily see that
\[l_2(z_i,z_j) = l_2(x_i,z_j) = 0\]
Using the relation in Equation \ref{eq:linking-symmetry}, we then see that
\begin{align*}
l_2(z_i,x_j) &= l_2(x_j,z_i) + \langle x_j, z_j \rangle \\
&= 0 + \langle x_j, -x_i - Qy_i \rangle\\
&= \langle x_j, -Q_{i,j}y_j \rangle\\
&= - Q_{i,j}
\end{align*}
Furthermore, we have that $z_i + x_i = -Qy_i$.  Consequently
\[ l_2(z_i + x_i,z_j + x_j) =  0\]
Using the bilinearity relation, we obtain
\[l_2(z_i,z_j) + l_2(z_i,x_j) + l_2(x_i,z_j) + l_2(x_i,x_j) = 0\]
Using the above formulas, we see that
\[l_2(x_i,x_j) = -l_2(z_i,x_j) = Q_{i,j}\]
Next, we have
\[- Q_{i,j} = l_2(z_i,x_j) = l_2(-x_i,x_j) + l_2(-Qy_i,x_j) = -Q_{i,j} + l_2(Qy_i,x_j)\]
so $ l_2(Qy_i,x_j) = 0$ for all $i,j$.  This implies that $l_2(y_i,x_j) = 0$ and consequently $l_2(x_j,y_i) = \langle y_i,x_j \rangle = \delta_{i,j}$.
\end{proof}

\begin{proposition}
The linking form $l_3$ satisfies
\begin{align*}
l_3(z_i,z_j) &= 0&\ l_3(z_i,x_j) &= 0\\
l_3(x_i,z_j) &= Q_{i,j} & l_3(x_i,x_j) &= 0 \\
l_3(x_i,y_j) &= -\delta_{i,j} & l_3(y_i,y_j) &= Q^{-1}_{i,j}
\end{align*}
\end{proposition}

\begin{proof}
As in the proof of Proposition \ref{prop:linking-l2}, we have that
\[l_3(z_i,A) = l_3(A,x_j) = 0\]
for any $A \in H_1(\Sigma;\ZZ)$.  We then easily obtain
\[l_3(z_i,z_j) = l_3(z_i,x_j) = l_3(x_i,x_j) = 0\]
It then follows that
\[l_3(x_i,z_j) = \langle z_j,x_i \rangle = Q_{i,j}\]
Now,
\[Q_{i,j} = l_3(x_i,z_j) = l_3(x_i,-x_j) + l_3(x_i,-Qy_j) = l_3(x_i,-Qy_j)\]
which implies that $l_3(x_i,y_j) = -\delta_{i,j}$.  Finally,
\begin{align*}
l_3(z_i,z_j) &= l_3(-x_i,-x_j) + l_3(-x_i,-Qy_j) + l_3(-Qy_i,-x_j) + l_3(-Qy_i,-Qy_j) \\
&= 0 - Q_{i,j} + 0 + l_3(Qy_i,Qy_j) \\
&= 0
\end{align*}
Thus $ l_3(Qy_i,Qy_j) = Q_{i,j}$.  Since $Q$ is the matrix for the linking form in the basis $\{Qy_i\}$, we see that $Q^{-1}$ is the matrix for the linking form in the basis $\{y_i\}$.
\end{proof}

\section{Rohlin's Theorem}

\begin{lemma}
\label{lemma:spin-cobordant}
Let $X$ be a closed, spin 4-manifold.  Then $X$ is spin-cobordant to a 4-manifold $X'$ such that
\begin{enumerate}
\item $X'$ has a handle decomposition without 1- or 3-handles, and
\item the intersection form of $X'$ is even and indefinite.
\end{enumerate}
In particular, $X'$ is spin and satisfies $\sigma(X') = \sigma(X)$.
\end{lemma}

\begin{proof}
Following the standard surgery theory trick, we can attach 5-dimensional 2-handles to kill the 1-handles of $X$.  Moreover, we can extend the spin structure across these 2-handles.  Turning the manifold upside down, we can also kill all of the 3-handles.  This results in some spin $X''$ that has no 1- or 3-handles.  By attaching one more 2-handle along a contractible loop, we get a spin cobordism to $X'' \# S^2 \times S^2$, which has an even, indefinite intersection form. 
\end{proof}

The connection between handle decompositions and trisections is given by the following result.

\begin{theorem}[\cite{GK}]
\label{thrm:trisection-existence}
Let $X$ be a closed, oriented 4-manifold that admits a handle decomposition with a single $0$-handle, $k_1$ 1-handles, $n$ 2-handles, and $k_3$ 3-handles.  Then $X$ admits a $(g;k_1,g-n,k_2)$ trisection for some $g$.
\end{theorem}

\begin{corollary}
\label{cor:trisection-gsc}
Suppose that a closed, smooth oriented 4-manifold $X$ admits a handle decomposition without 1- or 3-handles.  Then 
\begin{enumerate}
\item $X$ admits a $(g;0,k,0)$ trisection for some $g,k$, and
\item $X$ admits a $(g;k,0,0)$ trisection for some $g,k$.
\end{enumerate}
\end{corollary}

\begin{proof}
The first statement follows directly from Theorem \ref{thrm:trisection-existence}, while the second follows from the first by cyclically permuting the sectors of the trisection.
\end{proof}

A pseudotrisection with intersection form $E_8$ is guaranteed by Proposition \ref{prop:pseudo-Q}.  An example of such a pseudotrisection (which is not obtained by the method of the proof of Proposition \ref{prop:pseudo-Q}, however), is given in Figure \ref{fig:pseudoE8}.

%%%%%%%%%%%%%%%%%%%%%%%%%%%%%%%%%%%%%%%%%%%%%%%
\begin{figure}[h!]
\centering
\includegraphics[width=.95\textwidth]{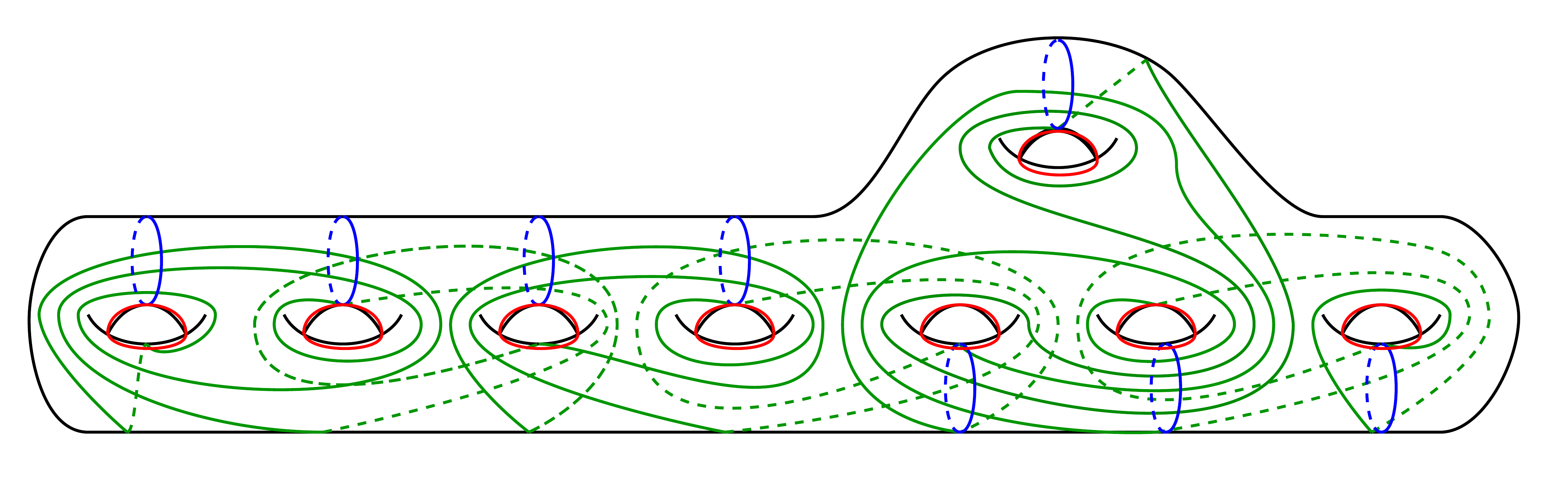}
\caption{A pseudotrisection with intersection form $E_8$.}
\label{fig:pseudoE8}
\end{figure}
%%%%%%%%%%%%%%%%%%%%%%%%%%%%%%%%%%%%%%%%%%%%%%%

\begin{proposition}
Let $(\Sigma,\alphas,\betas,\gammas)$ be the Heegaard triple in Figure \ref{fig:pseudoE8}.  Then $Y_{\alpha \beta} \cong Y_{\gamma \alpha} \cong S^3$ and $Y_{\beta \gamma}$ is the Poincar\'e homology sphere.
\end{proposition}

\begin{proof}
Each $\alpha_i$ intersects a unique $\beta$-curve in a single point and a unique $\gamma$ curve in a single point.  This implies that the pairs $(\Sigma,\alpha,\beta)$ and $(\Sigma,\gamma,\alpha)$ are the standard Heegaard diagram for $S^3$.

Viewing $\Sigma$ as a Heegaard surface in $Y_{\alpha \beta}$, the surface framing of each $\gamma$ curve is $+2$ and the curves $\{\gamma_1,\dots,\gamma_8\}$ are a collection of unknots that link according to the $E_8$ Dynkin diagram.  It is well-known (e.g. \cite{KS-8faces}) that the result of surgery on this framed link results in the Poincar\'e homology sphere.
\end{proof}

We now have sufficient tools to prove Rohlin's Theorem

\begin{proof}[Proof of Theorem \ref{thrm:gsc-Rohlin}]

Fix $m,n,k \geq 0$ and set $g = 8m+2n+k$.  Then we can obtain a $(g;k,0,0)$ pseudotrisection $(\Sigma,\alphas,\betas,\gammas)$ with intersection form $Q = m E_8 \oplus n H$ by taking a connected sum of $m$ copies of the counterfeit $E_8$, $n$ copies of the standard trisection of $S^2 \times S^2$, and $k$ copies of the standard $(1;1,0,0)$ trisection of $S^4$.

Let $\phi \in \cK_g$ be any element of the Johnson kernel and let  $(\Sigma,\alphas,\betas,\phi(\gammas))$ be the resulting pseudotrisection.  We obtain 3-manifolds $Y_{2,\phi}$ and $Y_{3,\phi}$ as the union of handlebodies.  Since $\phi \in \cK_g$, it is the product of $r$ separating twists.  Consequently, $Y_{2,\phi}$ is obtained from $Y_2$ by surgery on a boundary link $L_2 = l_{2,1} \cup \dots \cup l_{2,r}$ and $Y_{3,\phi}$ is obtained from $Y_3$ by surgery on a boundary link $L_3 = l_{3,1} \cup \dots \cup l_{3,r}$.  The surgery formula for the Casson invariant implies that
\begin{align*}
\mu(Y_{2,\phi}) - \mu(Y_2) &= \sum_{i = 1}^r \text{Arf}(l_{2,i}) \\
\mu(Y_{3,\phi}) - \mu(Y_3) &= \sum_{i = 1}^r \text{Arf}(l_{3,i}) \\
\end{align*} 
But by Lemma \ref{lemma:same-Arf}, we have that $\text{Arf}(l_{2,i}) = \text{Arf}(k_{3,i})$ for all $i = 1,\dots,r$.  Consequently
\[ \mu(Y_{2, \phi}) + \mu(Y_{3,\phi}) = \mu(Y_2) + \mu(Y_3) = m \text{ mod }2.\]
Now, suppose that $X$ is any closed, smooth, oriented 4-manifold with intersection $Q$ and a $(8m+2n+k;k,0,0)$ trisection.  By Theorem \ref{thrm:torelli-trisection}, we can assume there is some element $\phi \in \cK_g$ such that $(\Sigma_g,\alphas,\betas,\phi(\gammas))$ is a trisection diagram for $X$.  Moreover, since this is an honest trisection decomposition, we have that $Y_{2,\phi} \cong Y_{3,\phi} \cong S^3$.  Therefore
\[m = \mu(Y_{2,\phi}) + \mu(Y_{3,\phi}) = 0 \text{ mod }2\]
\end{proof}

%%%%%%%%%%%%%%%%%%%%%%%%%%%%%%%%%%%%%%%%%%%%%%%%%%%%%%%
\bibliographystyle{alpha}
\nocite{*}
\bibliography{References}

%%%%%%%%%%%%%%%%%%%%%%%%%%%%%%%%%%%%%%%%%%%%%%%%%%%%%%%

\end{document}